\documentclass[reqno,12pt]{amsart}
\usepackage[margin=1.1in, footskip=1cm]{geometry}
\usepackage{amsmath, amsthm, amssymb, booktabs}
\usepackage{mathrsfs}
\pdfpagewidth=\paperwidth
\pdfpageheight=\paperheight

\newtheorem{theorem}{Theorem}[section]
\newtheorem{lemma}[theorem]{Lemma}

\theoremstyle{definition}

\theoremstyle{remark}

\numberwithin{equation}{section}

\newcommand{\mmod}[1]{\,\,({\rm{mod}}\,\,#1)}

\def\dbN{{\mathbb N}}

 \def\bfalp{{\boldsymbol \alpha}}

\def\le{\leqslant} \def\ge{\geqslant}

\makeatletter
\@namedef{subjclassname@2020}{\textup{2020} Mathematics Subject Classification}
\makeatother

\begin{document}
\title[Well-distribution modulo one and the primes]{Well-distribution modulo one and the primes}
\author[J. Champagne]{J. Champagne}
\address{JC and Y-RL: Department of Pure Mathematics, University of Waterloo, 200 University Avenue West, Waterloo, ON, N2L 3G1, Canada}
\email{jchampag@uwaterloo.ca, yrliu@uwaterloo.ca }
\author[T. H. L\^e]{T. H. L\^e}
\address{THL: Department of Mathematics, University of Mississippi, 305 Hume Hall, University, MS 38677, USA}
\email{leth@olemiss.edu}
\author[Y.-R. Liu]{Y. R. Liu}
\author[T. D. Wooley]{T. D. Wooley}
\address{TDW: Department of Mathematics, Purdue University, 150 N. University Street, West 
Lafayette, IN 47907-2067, USA}
\email{twooley@purdue.edu}
\thanks{The second author is supported by NSF grant DMS-2246921 and a travel award from the Simons Foundation. The third 
author is supported by an NSERC discovery grant. The fourth author is supported by NSF grants DMS-1854398 and DMS-2001549.}
\date{}

\begin{abstract} Let $(p_n)$ denote the sequence of prime numbers, with $2=p_1<p_2<\ldots$. We demonstrate the existence of an 
irrational number $\alpha$ having the property that the sequence $(\alpha p_n)$ is not well-distributed modulo $1$.
\end{abstract}

\maketitle

\section{Introduction}
Consider a real sequence $(s_n)$ and the associated fractional parts $\{ s_n\}=s_n-\lfloor s_n\rfloor$. This sequence is said to be 
{\it equidistributed} (or {\it uniformly distributed}) modulo $1$ when, for each pair $a$ and $b$ of real numbers with $0\le a<b\le 1$, 
one has
\[
\lim_{N\rightarrow \infty}\frac{\text{card} \{ n\in [1,N]\cap \mathbb Z: a\le \{s_n\}\le b\}}{N}=b-a.
\]
A stronger notion than equidistribution is obtained by insisting that, for each natural number $m$, the sequence $(s_{n+m})$ should 
be equidistributed modulo $1$, uniformly in $m$. This property of being {\it well-distributed} modulo $1$ was introduced by 
Petersen \cite{Pet1956} in 1956. More concretely, we say that the sequence $(s_n)$ is well-distributed modulo $1$ when, for each 
pair $a$ and $b$ with $0\le a<b\le 1$, one has
\[
\lim_{N\rightarrow \infty}\sup_{m\in \mathbb N}\Bigl| \frac{\text{card}\{n\in [1,N]\cap \mathbb Z:a\le \{s_{n+m}\}\le b\}}{N}-(b-a)\Bigr| =0.
\] 
It is a consequence of pioneering work of Weyl \cite{Wey1916} that, given a polynomial
\[
\psi_d(t;\bfalp)=\alpha_dt^d+\ldots +\alpha_1t+\alpha_0\in \mathbb R[t]
\]
with an irrational coefficient $\alpha_l$ for some $l\ge1$, the sequence $(\psi_d(n;\bfalp))$ is equidistributed modulo $1$. Moreover, 
Lawton \cite[Theorem 2]{Law1959} established that, under the same conditions, this sequence satisfies the stronger property of being 
well-distributed modulo $1$. We refer the reader to Bergelson and Moreira \cite[\S3]{BL2016} for further discussion on sequences well-distributed modulo $1$.\par

In this note, we focus on the sequence $(p_n)$ of prime numbers with $2=p_1<p_2<\ldots $. It was famously proved by Vinogradov 
that when $\alpha$ is irrational, then the sequence $(\alpha p_n)$ is equidistributed modulo $1$ (see \cite{Vin1942}, for example). 
Subsequently, a relatively simple proof of this conclusion was presented by Vaughan \cite{Vau1977}. It is natural to enquire whether 
this equidistribution extends to a corresponding well-distribution property. The purpose of this note is to answer this question in the 
negative.

\begin{theorem}\label{theorem1.1}
There exists an irrational number $\alpha$ having the property that the sequence $(\alpha p_n)$ is not well-distributed modulo $1$.
\end{theorem}

At first sight, this conclusion may seem surprising, since the primes are undeniably equidistributed at large scales. However, as first 
discerned by Maier \cite{Mai1985}, the primes exhibit irregularities in their distribution at very small scales. When it comes to 
generating a failure of well-distribution, the most convenient manifestation of such irregularities of which to avail oneself is that 
established in work of Shiu \cite[Theorem 1]{Shi2000}. The latter author shows, in particular, that for each natural number $q$, there 
exist arbitrarily long strings of prime numbers $p_{n+1},\ldots ,p_{n+k}$ with 
$p_{n+1}\equiv \ldots \equiv p_{n+k}\equiv 1\mmod{q}$. By carefully constructing an associated irrational number $\alpha$, this 
congruential bias amongst consecutive primes may be shown to generate a corresponding failure of well-distribution modulo $1$ in the 
sequence $(\alpha p_n)$. It will be evident from our proof of Theorem \ref{theorem1.1}, which we present in \S2, that many such 
numbers $\alpha$ can be constructed, each of which is transcendental.\par

As is usual, we write $e(z)$ for $e^{2\pi{\rm i}z}$, and for $\theta\in \mathbb R$ we define 
$\|\theta\|=\min \{ |\theta -t|:t\in \mathbb Z\}$.

\section{The application of Shiu's theorem}
Our strategy for proving Theorem \ref{theorem1.1} depends on the careful selection of a sequence $(n_k)$ of natural numbers with 
$1\le n_0<n_1<\ldots$, and the associated real number
\begin{equation}\label{2.1}
\alpha =\sum_{k=0}^\infty 2^{-n_k}.
\end{equation}
The sequence $(n_k)$ is defined iteratively in terms of a consequence of Shiu's theorem on strings of congruent primes. Thus, for 
each $n\in \mathbb N$, there exists an integer $m=m(n)$ with
\begin{equation}\label{2.2}
p_{m+1}\equiv \ldots \equiv p_{m+n}\equiv 1\mmod{2^n}.
\end{equation}
The conclusion of \cite[Theorem 1(i)]{Shi2000} shows that, when $n$ is sufficiently large, such an integer $m(n)$ exists with 
$m(n)<\exp_4(n)$, where $\exp_r(x)$ denotes the $r$-fold iterated exponential function. We define the sequence $(n_k)$ as follows. 
We put $n_0=1$, and then define
\[
m_k=m(n_k),\quad \pi_k=p_{m_k+n_k},\quad n_{k+1}=4\pi_k\quad (k\ge 0).
\]

\par We investigate the well-distribution of the sequence $(\alpha p_n)$ by means of an analogue of Weyl's criterion for 
equidistribution. Thus, as a consequence of \cite[Theorems 2 and 3]{Pet1956}, we see that $(\alpha p_n)$ is well-distributed modulo 
$1$ if and only if, for each $h\in \mathbb N$, one has
\begin{equation}\label{2.3}
\lim_{N\rightarrow \infty}\sup_{m\in \mathbb N}\biggl| N^{-1}\sum_{n=1}^{N}e(h\alpha p_{n+m})\biggr| =0.
\end{equation}

\par Before embarking on the proof of Theorem \ref{theorem1.1} in earnest, we pause to confirm that the real number $\alpha$ 
defined in \eqref{2.1} is irrational.

\begin{lemma}\label{lemma2.1}
The number $\alpha$ is transcendental, and hence is irrational.
\end{lemma}

\begin{proof}
For each $k\in \mathbb N$, put
\[
q_k=2^{n_k}\quad \text{and}\quad a_k=2^{n_k}\sum_{l=0}^k 2^{-n_l}.
\]
Then we see that $a_k\in \dbN$ and $(q_k,a_k)=1$. Moreover, one has
\[
|q_k\alpha -a_k|\le \sum_{l>k}2^{n_k-l_k}.
\]
It is evident from the definition of the function $m(n)$ via \eqref{2.2} that $\pi_k>2^{n_k}$, whence $n_{k+1}>2^{n_k}$. 
Consequently, whenever $k$ is large enough and $l>k$, one has
\[
0<\sum_{l>k}2^{n_k-n_l}<\sum_{l>k}2^{-l-kn_k}<q_k^{-k}.
\]
We therefore deduce that $|q_k\alpha -a_k|<q_k^{-k}$. By Liouville's theorem (see \cite[Theorem 1.1]{Bak2022}, for example), it 
therefore follows that $\alpha$ cannot be algebraic. Thus, we conclude that $\alpha $ is transcendental, and hence irrational.
\end{proof}

We next show that, for each positive integer $h$, the real number $\| h\alpha (p_n-1)\|$ is infinitely often very small for long strings 
of consecutive primes $p_n$.

\begin{lemma}\label{lemma2.2}
Let $h$ be a positive integer. Then for each sufficiently large positive integer $k$, one has
\[
\| h\alpha (p_{i+m_k}-1)\|<\pi_k^{-2}\quad (1\le i\le n_k).
\]
\end{lemma}

\begin{proof} Given a positive integer $h$, when $k$ is sufficiently large and $1\le i\le n_k$, one has
\[
0<h(p_{i+m_k}-1)\sum_{l>k}2^{-n_l}\le \sum_{l>k}h\pi_k 2^{-l-3\pi_k}<\pi_k^{-2}.
\] 
Meanwhile, under the same conditions, one finds that since $p_{i+m_k}\equiv 1\mmod{2^{n_k}}$, then
\[
h(p_{i+m_k}-1)\sum_{l=0}^k2^{-n_l}\equiv 0\mmod{1}.
\]
By combining these conclusions, therefore, we infer that for $1\le i\le n_k$, one has
\[
\biggl\| h(p_{i+m_k}-1)\biggl( \sum_{l=0}^k2^{-n_l}+\sum_{l>k}2^{-n_l}\biggr) \biggr\| <\pi_k^{-2},
\]
and the desired conclusion follows from \eqref{2.1}.
\end{proof}

We are now equipped to complete the proof of Theorem \ref{theorem1.1}. Let $h$ be a natural number. From Lemma 
\ref{lemma2.2}, we find that whenever $k$ is sufficiently large and $1\le i\le n_k$, one has
\[
|e(h\alpha p_{i+m_k})-e(h\alpha)|=|e(h\alpha (p_{i+m_k}-1))-1|<\pi_k^{-1}.
\]
Thus, when $N$ is an integer with $1\le N\le n_k$, then
\[
\biggl| N^{-1}\sum_{n=1}^Ne(h\alpha p_{n+m_k})-N^{-1}\sum_{n=1}^Ne(h\alpha )\biggr| <\pi_k^{-1},
\]
whence
\[
1-\pi_k^{-1}<\biggl| N^{-1}\sum_{n=1}^Ne(h\alpha p_{n+m_k})\biggr|\le 1.
\] 
In particular, for each positive integer $N$, we infer that
\[
1-\frac{1}{N}\le \sup_{m\in \mathbb N}\biggl| N^{-1}\sum_{n=1}^Ne(h\alpha p_{n+m})\biggr|\le 1.
\]
This relation confirms that
\begin{equation}\label{2.4}
\lim_{N\rightarrow \infty}\sup_{m\in \mathbb N}\biggl| N^{-1}\sum_{n=1}^Ne(h\alpha p_{n+m})\biggr| =1,
\end{equation}
in contradiction with Weyl's criterion for well-distribution modulo $1$ given in \eqref{2.3}. We are therefore forced to conclude that $(\alpha p_n)$ is not 
well-distributed modulo $1$. In view of Lemma \ref{lemma2.1}, this completes the proof of Theorem \ref{theorem1.1}.\par

Although the case $h=1$ of \eqref{2.4} suffices to prove Theorem \ref{theorem1.1}, we gave a more general argument since it might be useful. We note also that the number 
$\alpha$ may be modified extensively without impairing the validity of our proof. Indeed, given an integer $q\ge 2$ and a sequence of positive integers $(b_k)$ not growing 
too rapidly, the number $\alpha$ defined in \eqref{2.1} could be replaced by
\[
\beta=\sum_{k=0}^\infty b_k q^{-n_k},
\]
and still the sequence $(\beta p_n)$ is not well-distributed modulo $1$. Furthermore, the rapid growth of the integer $n_k$ may be considerably weakened without damaging 
the crude bound of Lemma \ref{lemma2.2}, and so the Liouville-type properties of $\alpha$ may also be relaxed. 

\bibliographystyle{amsbracket}
\providecommand{\bysame}{\leavevmode\hbox to3em{\hrulefill}\thinspace}

\end{document}